\documentclass[11pt,reqno,a4paper]{amsart}

\begin{document}

\newcommand{\C}{\mathbb C}
\newcommand{\N}{\mathbb N}
\newcommand{\R}{\mathbb R}

\renewcommand{\Im}{{\mathrm{Im}}}
\renewcommand{\Re}{{\mathrm{Re}}}

\def\qqs{\ \forall \ }
\def \sprd#1,#2,#3{{\left( #1,#2\right)}_{#3}}
\newcommand{\lmd}{{\mathcal L}}

\newtheorem{theo}{Theorem}[section]
\newtheorem{prop}[theo]{Proposition}
\newtheorem{lem}[theo]{Lemma}
\newtheorem{cor}[theo]{Corollary}
\theoremstyle{definition}
\newtheorem{defi}[theo]{Definition}
\theoremstyle{remark}
\newtheorem{rem}[theo]{Remark}

\renewcommand{\cdots}{\dots}

\renewcommand{\theenumi}{\roman{enumi}}
\renewcommand{\labelenumi}{\theenumi)}

\title{The Influence of the Tunnel Effect on $L^{\infty}$-time decay}

\author{F. Ali Mehmeti}

\address{%
Univ Lille Nord de France, F-59000 Lille, France \newline
\indent
 UVHC, LAMAV, FR CNRS 2956, F-59313 Valenciennes, France}

\email{felix.ali-mehmeti@univ-valenciennes.fr}

\author{R. Haller-Dintelmann}

\address{%
TU Darmstadt \\
Fachbereich Mathematik \\
Schlo{\ss}gartenstra{\ss}e 7\\
64289 Darmstadt\\
Germany}

\email{haller@mathematik.tu-darmstadt.de}

\author{V. R\'egnier}

\address{%
Univ Lille Nord de France, F-59000 Lille, France \newline
\indent
 UVHC, LAMAV, FR CNRS 2956, F-59313 Valenciennes, France}

\email{Virginie.Regnier@univ-valenciennes.fr}

\begin{abstract}
We consider the Klein-Gordon equation on a star-shaped network
composed of $n$ half-axes connected at their origins. We add a
potential which is constant but different on each branch. Exploiting
a spectral theoretic solution formula from a previous paper, we
study  the $L^{\infty}$-time decay via H\"ormander's version of the
stationary phase method. We analyze the coefficient $c$ of the
leading term $c\cdot t^{-1/2}$ of the asymptotic expansion of the
solution with respect to time. For two branches we prove that for an
initial condition in an energy band above the threshold of tunnel
effect, this coefficient tends to zero on the branch with the higher
potential, as the potential difference tends to infinity. At the
same time the incline to the $t$-axis and the aperture of the cone
of $t^{-1/2}$-decay in the $(t,x)$-plane tend to zero.
\end{abstract}



\thanks{Parts of this work were done, while the second author visited
the University of Valenciennes. He wishes to express his gratitude to
F. Ali Mehmeti and the LAMAV for their hospitality}

\maketitle

%
%
%
%
\section{Introduction}
%
%
%
%
\noindent In this paper we study the $L^{\infty}$-time decay of
waves in a star shaped network of one-dimensional semi-infinite
media having different dispersion properties.   Results in
experimental physics \cite{Nim, hai.nim}, theoretical physics
\cite{D-L} and functional analysis \cite{Alreg3,yas} describe
phenomena created in this situation by the dynamics of the tunnel
effect: the delayed reflection and advanced transmission near nodes
issuing two branches. Our purpose is to describe the influence of
the height of a potential step on the $L^{\infty}$-time decay of
wave packets   above the threshold of tunnel effect, which sheds a
new light on its dynamics.

In this proceedings contribution we state results for a special
choice of initial conditions. The proofs in a more general context
will be the core of another paper.

The dynamical problem can be described as follows:

\noindent
Let  $N_1, \cdots, N_n$ be $n$ disjoint copies of $(0,+ \infty)$
with $n \geq 2$. Consider numbers  $a_k, c_k$ satisfying $0 < c_k$,
for $k = 1, \dots, n$ and $0 \leq a_1 \leq a_2 \leq \ldots \leq a_n
< + \infty$. Find a vector $(u_1, \dots, u_n)$ of functions $u_k:
[0, +\infty) \times \overline{N_k} \rightarrow \C$ satisfying the
Klein-Gordon equations
\[ [\partial_t^2 - c_k \partial_x^2   + a_k ] u_k(t,x) = 0 , \ k = 1, \dots, n,
\]
on $N_1, \cdots, N_n$ coupled at zero by usual Kirchhoff conditions
and complemented with initial conditions for the functions $u_k$ and their
derivatives.

Reformulating this as an abstract Cauchy problem, one is confronted with the
self-adjoint operator $A = (-c_k \cdot \partial^2_x + a_k)_{k=1, \dots, n}$
in $\prod_{k=1}^n L^2(N_k)$, with a domain that incorporates the Kirchhoff transmission
conditions at zero. For an exact definition of $A$, we refer to
Section~\ref{sec:sol:form}.

Invoking functional calculus for this operator, the solution can be given in
terms of
\[e^{\pm i \sqrt{A}t}u_0 \hbox{ and }e^{\pm i \sqrt{A}t}v_0.
\]
In a previous paper (\cite{ahr.arxiv2}, see also \cite{ahr.taranto})
we construct explicitly a spectral represen\-tation of \\
$\prod_{k=1}^n L^2(N_k)$ with respect to $A$ involving $n$ families
of generalized eigenfunctions. The $k$-th family is defined on
$[a_k,\infty)$ which reflects that $\sigma(A) = [a_1,\infty)$ and
that the multiplicity of the spectrum is $j$ in $[a_j,a_{j+1}),
j=1,\ldots, n$, where $ a_{n+1} = + \infty$. In this band $(a_j,
a_{j+1})$ the generalized eigenfunctions exhibit exponential decay
on the branches $N_{j+1},\ldots, N_n$, a fact called "multiple
tunnel effect" in \cite{ahr.arxiv2}.

In Section~\ref{sec:sol:form} we recall the solution formula proved
in \cite{ahr.arxiv2}. In Section~\ref{decay} we use H\"ormander's
version of the stationary phase method to derive the leading term of
the asymptotic expansion of the solution on certain branches and for
initial conditions in a compact energy band included in
$(a_j,a_{j+1})$. We obtain $c\cdot t^{-1/2}$ in cones in the
$(t,x)$-space delimited by the group velocities of the limit
energies and the dependence of $c$ on the coefficients of the
operator is indicated.   One can prove that outside these cones the
$L^{\infty}$-norm decays at least as $t^{-1}$. The complete analysis
will be carried out in a more detailed paper.

For the case of two branches and wave packets having a compact
energy band included in $(a_2,\infty)$, we show in
Section~\ref{acc.decay} that $c$ tends to zero on the side of the
higher potential, if $a_1$ stays fixed and $a_2$ tends to infinity.
We observe further that the exact $t^{-1/2}$-decay takes place in a
cone in the $(t,x)$-plane whose aperture and incline to the $t-$axis
tend to zero as $a_2$ tends to infinity. Physically the model
corresponds to a relativistic particle in a one dimensional world
with a potential step of amount $a_2-a_1$ in $x=0$. Our result
represents thus a dynamical feature for phenomena close to tunnel
effect, which might be confirmed by physical experiments.

Our results are designed to serve as tools in some pertinent
applications as the study of more general networks of wave guides
(for example microwave networks \cite{poz}) and the treatment of
coupled transmission conditions \cite{cm.cbc}.

For the Klein-Gordon equation in $\R^n$ with constant coefficients
the $L^{\infty}$-time decay   $c\cdot t^{-1/2}$ has been proved in
\cite{stramawa}. Adapting their method to a spectral theoretic
solution formula for two branches, it has been shown in
\cite{fam2,fam3} that the $L^{\infty}$-norm decays  at least as
$c\cdot t^{-1/4}$.

In \cite{lie.opt} and several related articles, the author studies
the $L^{\infty}$-time decay for crystal optics using similar methods.

In \cite{kostr}, the authors consider general networks with
semi-infinite ends. They give a construction to compute some
generalized eigenfunctions but no attempt is made to construct
explicit inversion formulas. In \cite{vbl.evl} the relation of the
eigenvalues of the Laplacian in an $L^{\infty}$-setting on infinite,
locally finite networks to the adjacency operator of the network is
studied.

\vspace{5mm}

 \noindent
{\it Acknowledgements}

\medskip

The authors thank Otto {\sc Liess} for useful remarks.

%
%
%
%
\section{A solution formula} \label{sec:sol:form}
%
%
%
%
\noindent The aim of this section is to recall the tools we used in
\cite{ahr.arxiv2} as well as the solution formula of the same paper
for a special initial condition and to adapt this formula for the
use of the stationary phase method
in the next section.
\begin{defi}[Functional analytic framework]  \label{def.A} ~
\begin{enumerate}
\item  Let $n \geq 2$ and  $N_1, \cdots, N_n$ be $n$ disjoint sets
identified with $(0, +\infty)$. Put $N := \bigcup_{k=1}^n
      \overline{N_k}$, identifying the endpoints $0$.
\\
      For the
      notation of functions two viewpoints are useful:
      \begin{itemize}
      \item functions $f$ on the object $N$ and $f_k$ is the restriction of
        $f$ to $N_k$.
      \item $n$-tuples of functions on the branches $N_k$; then sometimes we
            write $f = (f_1, \cdots, f_n)$.
      \end{itemize}
\item Two transmission conditions are introduced:
      \begin{align*}
        \strut \text{($T_0$): } & (u_k)_{k=1,\dots,n} \in
               \prod_{k=1}^n C^(\overline{N_k}) \text{ satisfies } u_i(0) =
               u_k(0), \ i,k \in \{ 1, \cdots, n \}.
        \intertext{This condition in particular implies that $(u_k)_{k=1, \dots, n}$ may
            be viewed as a well-defined function on $N$.}
        \text{($T_1$): } & (u_k)_{k=1,\dots,n} \in \prod_{k=1}^n
               C^1(\overline{N_k}) \text{ satisfies } \sum_{k=1}^n c_k \cdot
               \partial_x u_k(0^+) = 0.
      \end{align*}
\item Define the real Hilbert space $H = \prod_{k=1}^n L^2(N_k)$ with scalar product
      \[ (u,v)_H = \sum_{k=1}^n (u_k,v_k)_{L^2(N_k)}
      \]
      and the operator $A: D(A) \longrightarrow H$ by
      \[ \begin{aligned}
            D(A) &= \Bigl\{ (u_k)_{k=1, \dots, n} \in \prod_{k=1}^n H^2(N_k) :
        (u_k)_{k=1, \dots, n} \text{ satisfies } (T_0) \text{ and }
        (T_1) \Bigr\}, \\
            A((u_k)_{k=1, \cdots, n}) &= (A_ku_k)_{k=1, \cdots, n} = (-c_k \cdot \partial^2_x
            u_k + a_k u_k
                  )_{k=1, \cdots, n}.
         \end{aligned}
      \]
\end{enumerate}
\end{defi}

Note that, if $c_k = 1$ and $a_k=0$ for every $k \in \{1, \cdots, n
\}$, $A$ is the Laplacian in the sense of the existing literature,
cf. \cite{vbl.evl, kostr}. \rm
%
%
%
\noindent
\begin{defi}[Fourier-type transform $V$]  \label{def.V} ~
\begin{enumerate}
\item For $k \in \{ 1, \dots, n\}$ and $\lambda \in \C$ let
\[ \xi_k(\lambda) := \sqrt{\frac{\lambda - a_k}{c_k}} \quad \text{and} \quad
    s_k := - \frac{\sum_{l \neq k} c_l \xi_l(\lambda)}{c_k \xi_k(\lambda)}.
\]
Here, and in all what follows, the complex square root is chosen in such a way
that $\sqrt{r \cdot e^{i \phi}} = \sqrt{r} e^{i \phi/2}$ with $r>0$ and $\phi
\in [-\pi,\pi)$.\\
\item For $\lambda \in \C$ and $j,k \in \{ 1, \cdots, n\}$, we
define generalized eigenfunctions $F_{\lambda}^{\pm,j}: N
\rightarrow \C$ of $A$ by $F_{\lambda}^{\pm,j}(x) :=
F_{\lambda,k}^{\pm,j}(x)$ with
\[ \left\{ \begin{aligned}
       F_{\lambda,k}^{\pm,j}(x) &= \cos(\xi_j( \lambda) x ) \pm i
       s_j( \lambda)   \sin(\xi_j( \lambda) x ), &  \text{for } k = j,\\
      F_{\lambda,k}^{\pm,j}(x) &= \exp(\pm i \xi_k(\lambda)x), &
          \text{for } k \neq j.
   \end{aligned} \right.
\]
for $x \in \overline{N_k}.$
\item For $l = 1, \dots, n$ let
\[ q_l(\lambda) := \begin{cases}
            0, & \text{if } \lambda < a_l, \\
                    \frac{c_l \xi_l(\lambda)}{|\sum_{j=1}^n c_j \xi_j(\lambda)|^2}, & \text{if } a_l <
            \lambda.
                   \end{cases}
\]
\item Considering for every $k = 1, \dots, n$ the weighted space $L^2((a_k, +\infty),
q_k)$, we set $L^2_q := \prod_{k=1}^n L^2((a_k, + \infty),q_k)$. The
corresponding scalar product is
\[ (F, G)_q := \sum_{k=1}^n \int_{(a_k, +\infty)} q_k(\lambda) F_k(\lambda)
    \overline{G_k(\lambda)} \; d \lambda
\]
and its associated norm $|F|_q := (F, F)_q^{1/2}$.
\item For all $f \in L^1(N, \C)$ we define $Vf: \Pi_{k=1}^n [a_k , + \infty) \to \C$ by
\[ (Vf)_k(\lambda):= \int_N f(x) \overline{(F_{\lambda}^{-,k})}(x) \; dx, \ k = 1, \ldots, n.
\]
\end{enumerate}
\end{defi}
%
In \cite{ahr.arxiv2}, we show that $V$ diagonalizes $A$ and we
determine a metric setting in which it is an isometry. Let us recall
these useful properties of $V$ as well as the fact that the
property $u \in D(A^j)$ can be characterized in terms of the decay
rate of the components of $Vu$.
\begin{theo} \label{V.iso}
Endow $\prod_{k=1}^n C_c^\infty(N_k)$ with the norm of $H =
\prod_{k=1}^n L^2(N_k)$. Then
\begin{enumerate}
\item  $V : \prod_{k=1}^n C_c^\infty(N_k) \to L^2_q$ is isometric and can
be extended to an isometry $\tilde{V} : H \to L^2_q$, which we shall
again denote by $V$ in the following.
\item $V : H \to L^2_q$ is a spectral representation of $H$ with respect to $A$.
In particular, $V$ is surjective.
\item The spectrum of the operator $A$ is $\sigma(A) = [a_1, + \infty)$.
\item For $l \in \N$ the following statements are equivalent:
    \begin{enumerate}
    \item $u \in D(A^l)$,
    \item $\lambda \mapsto \lambda^{l} (Vu)(\lambda) \in L_q^2$,
    \item $\lambda \mapsto \lambda^{l} (Vu)_k(\lambda) \in
    L^2((a_k, + \infty),q_k), \ k=1,\ldots,n$.
    \end{enumerate}
\end{enumerate}
\end{theo}
%
Denoting
$F_{\lambda}(x):=(F_{\lambda}^{-,1}(x), \ldots,
F_{\lambda}^{-,n}(x))^{T}$ and
$P_j = \left( \begin{array}{r|l}
I_j &  0 \\
\hline 0 &  0
\end{array}
\right)$,
where $I_j$ is the $j \times j$ identity matrix, for $\lambda \in
(a_j,a_{j+1})$ it holds:
$F_{\lambda}^T q(\lambda) F_{\lambda} = (P_j F_{\lambda})^T
q(\lambda) (P_j F_{\lambda})$ and
\begin{equation} \label{vector F}
P_j F_{\lambda} = \left( \begin{array}{c}
\left( +,*,*, \ldots, *, e^{-|\xi_{j+1}|x}, \ldots, e^{-|\xi_{n}|x} \right) \\
\left( *,+,*, \ldots, *, e^{-|\xi_{j+1}|x}, \ldots, e^{-|\xi_{n}|x} \right) \\
\left( *,*,+, \ldots, *, e^{-|\xi_{j+1}|x}, \ldots, e^{-|\xi_{n}|x} \right) \\
\vdots \\
\left( *,*, \ldots, *,+, e^{-|\xi_{j+1}|x}, \ldots, e^{-|\xi_{n}|x} \right) \\
0 \\
\vdots \\
0
\end{array} \right).
\end{equation}
Here \ $*$ \ means \ $e^{-i \xi_k(\lambda) }$ \  and \ $+$ \ means
\ $ \cos(\xi_k( \lambda) x ) - i s_k( \lambda) \sin(\xi_k(\lambda)
x) $ \
in the $k$-th column for $k=1,\ldots,j$.
 This can be interpreted as a multiple tunnel effect (tunnel
effect in the last $(n-j)$ branches with different exponential decay
rates). For $\lambda$ near $a_{j+1}$, the exponential decay of the
function $x \mapsto e^{-|\xi_{j+1}|x}$ is slow. The tunnel effect is
weaker on the other branches since the exponential decay is quicker.

We are now interested in the Abstract Cauchy Problem
\[ \hbox{(ACP)}: u_{tt}(t) + Au(t) = 0, \ t > 0, \textrm{ with } u(0)=u_0,\ u_t(0)=0.
\]
By the surjectivity of $V$ (cf. Theorem~\ref{V.iso}~(ii)) for every $j,k \in
\{1, \ldots, n\}$ with $k\leq j$ there exists an
initial condition $u_0 \in H$ satisfying

\medskip

\noindent
{\bf Condition} ${\mathbf(A_{j,k})}$:
  $(V u_0)_l \equiv 0, \ l \not= k$ and $(V u_0)_k \in C^2_c ((a_j, a_{j+1}))$.

\medskip

\begin{rem} \label{D A infty}
\begin{enumerate}
  \item We use the convention $a_{n+1}= + \infty$.
  \item For $u_0$ satisfying $(A_{j,k})$ there exist
  $a_j < \lambda_{\min} < \lambda_{\max}<a_{j+1}$
  such that
  \[\hbox{ supp} (V u_0)_k \subset [\lambda_{\min} , \lambda_{\max}]\]
  \item
  If $u_0 \in H$ satisfies $(A_{j,k})$, then
  $u_0 \in D(A^{\infty})= \displaystyle \bigcap_{l \geq 0} D(A^l)$,
  due to Theorem~\ref{V.iso}~(iv),
  since
  $\lambda \mapsto \lambda^{l} (Vu)_m(\lambda) \in
    L^2((a_m, + \infty),q_m), \ m=1,\ldots,n$
  for all $l \in \N$ by the compactness of
  $\hbox{ supp} (V u_0)_m$.
\end{enumerate}
\end{rem}
%
%
\begin{theo}[Solution formula of (ACP) in a special case]
\label{solform.lambda}
Fix $j,k \in \{1,\ldots,n\}$ with $k\leq j$. Suppose that  $u_0$
satisfies Condition $(A_{j,k})$. Then there exists a unique solution
$u$ of $(ACP)$ with
$u \in C^{l}([0;+ \infty), D(A^{m/2}))$
for all $l,m \in \N$. For $x \in N_r$ with $r \leq j$ such that
$r\neq k$
and $t \geq 0$, we have the representation
\begin{equation*}
u(t,x) = \frac{1}{2}(u_+(t,x) + u_-(t,x))
\end{equation*}
with
\begin{equation} \label{representation}
u_\pm(t,x):=\displaystyle\int_{\lambda_{\min}}^{\lambda_{\max}}
e^{\pm i \sqrt{\lambda}t} q_k(\lambda) e^{-i \xi_r(\lambda)x}
(Vu_0)_k(\lambda) d \lambda,
\end{equation}
\end{theo}
%
\begin{proof}
Since $v_0 = u_t(0) = 0$, we have for the solution of (ACP) the
representation
\[
u(t)=V^{-1} \cos(\sqrt{\lambda}t)Vu_0.
\]
(cf.~for example  \cite[Theorem 5.1]{fam2}). The expression for
$V^{-1}$ given in \cite{ahr.arxiv2} yields the formula for $u_\pm$.
\end{proof}
%
%
\begin{rem}
Expression \eqref{representation} comes from a term of the type $*$
in $F_\lambda$ (see \eqref{vector F}) via the representation of
$V^{-1}$. A solution formula for arbitrary initial conditions which
is valid on all branches is available in \cite{ahr.arxiv2}. This
general expression is not needed in the following.
\end{rem}
%
%
%
%
%
\section{$L^{\infty}$-time decay}
\label{decay}
The time asymptotics of the $L^{\infty}$-norm of the solution of
hyperbolic problems is an important qualitative feature, for example
in view of the study of nonlinear perturbations.

In \cite{miha} the author derives the spectral theory for the
3D-wave equation with different propagation speeds in two adjacent
wedges. Further he attempts to give the $L^{\infty}$-time decay
which he reduces to a 1D-Klein-Gordon problem with potential step
(with a frequency parameter). He uses interesting tools, but his
argument is technically incomplete: the backsubstitution (see the
proof of Theorem~\ref{sol.decay} below) has not been carried out,
and thus his results cannot be reliable. Nevertheless, we have been
inspired by some of his techniques.

The main problem to determine the $L^{\infty}$-norm is the
oscillatory nature of the integrands in the solution formula
\eqref{representation}. The stationary phase formula as given by
L.~H\"{o}rmander in Theorem~7.7.5 of \cite{horm} provides a powerful
tool to treat this situation.

In the following Theorem we formulate a special case of this result
relevant for us.
%
%
\begin{theo}[Stationary phase method] \label{horm.theo}
Let $K$ be a compact interval in $\R$, $X$ an open neighborhood of
$K$. Let $U \in C_0^2(K)$, $\Psi \in C^4(X)$ and $\Im \Psi \geq 0$
in $X$. If there exists $p_0 \in X$ such that
$ \frac{\partial}{\partial p}\Psi(p_0)=0,\
\frac{\partial^2}{\partial p^2}\Psi(p_0)\neq 0, $
and
$ \Im \Psi(p_0)=0, \ \frac{\partial}{\partial p}\Psi(p)\neq 0, \ p \in
K \setminus \{ p_0 \}, $
then
%
$$ \Big\vert
%
\int_K U(p) e^{i \omega \Psi(p)} dp \ - \ e^{i \omega \Psi(p_0)}
\left[ \frac{\omega}{2 \pi i} \frac{\partial^2}{\partial
p^2}\Psi(p_0) \right]^{-1/2} U(p_0)
\Big\vert
%
\leq
C(K) \  \| U \|_{C^2(K)}  \ \omega^{-1} \ .
$$
for all $\omega > 0$. Moreover $C(K)$ is bounded when $\Psi$ stays
in a bounded set in $C^4(X)$.
%
%
\end{theo}
%
%
%
%
\begin{theo}[Time-decay of the solution of (ACP) in a special case]
\label{sol.decay}
Fix $j,k \in \{1,\ldots,n\}$ with $k\leq j$. Suppose that  $u_0$
satisfies Condition $(A_{j,k})$ and choose $\lambda_{\min} ,
\lambda_{\max} \in (a_j , a_{j+1})$ such that
\[
\hbox{\rm supp} (V u_0)_k \subset [\lambda_{\min} , \lambda_{\max}]
\subset (a_j , a_{j+1}) .\
\]
Then for all  $x \in N_r$ with $r \leq j$ and $r \neq k$ and all
$t \in \R^+$ such that $(t,x)$ lies in the cone described by
\begin{equation}  \label{cone-t-x}
\sqrt{\frac{\lambda_{\max}}{c_r (\lambda_{\max}-a_r)}} \leq
\frac{t}{x} \leq
\sqrt{\frac{\lambda_{\min}}{c_r(\lambda_{\min}-a_r)}} \ ,
\end{equation}

\noindent there exists   $H(t,x,u_0) \in \C$ and a constant $c(u_0)$
satisfying
\begin{equation} \label{time decay}
\left | u_+(t,x) -  H (t,x,u_0) t^{-1/2} \right| \leq c(u_0) \cdot
t^{-1} \ ,
\end{equation}
where $u_+$ is defined in Theorem~\ref{solform.lambda}, with
\[ |  H (t,x,u_0) | \leq
\Bigl(\frac{2 \pi c_k}{c_r}\Bigr)^{1/2}
\lambda_{\max}^{3/4} \cdot \dfrac{\max_{v \in [v_{\min},v_{\max}]}
\sqrt{|(a_r-a_k)v+a_r|}} {\left( \sum_{l\leq r} \sqrt{c_l}
\sqrt{(a_r-a_l)v_{\min} + a_r} \right)^2}
\cdot  \| (V u_0)_k \|_{\infty},
\]
   where $v_{\min}:= \dfrac{a_r}{\lambda_{\max} - a_r}$ and
$v_{\max}:= \dfrac{a_r}{\lambda_{\min} - a_r}$.
\end{theo}
\begin{rem}
\begin{enumerate}
\item
Note that \eqref{cone-t-x} is equivalent to
\begin{equation}  \label{cone-v}
v_{\min} \leq v(t,x):= c_r (t/x)^2 -1 \leq v_{\max} \ .
\end{equation}
\item
The hypotheses of Theorem~\ref{sol.decay} imply that $j \geq 2$.
\item
An explicit expression for $ H (t,x,u_0)$ is given at the end of the
proof in \eqref{H}.
\item
We have chosen to investigate only $u_+$ in this proceedings
article, since the expression for $u_-$ does not posses a stationary
point in its phase. Hence, one can prove that its contribution will
decay at least as $c t^{-1}$. A detailed analysis will follow in a
forthcoming paper.
\end{enumerate}
\end{rem}

%
%
\begin{proof}
We devide the proof in five steps.
%
%
\paragraph{\it First step: Substitution}
Realizing the substitution $p:= \xi_r(\lambda)=\sqrt{\frac{\lambda -
a_r}{c_r}}$ in the expression for $u_+$ given in
Theorem~\ref{solform.lambda} leads to:
$$u_{+}(t,x) = 2 c_r \int_{p_{\min}}^{p_{\max}}
e^{ i \sqrt{a_r + c_r p^2}t} q_k(a_r + c_r p^2) e^{-i px}
(Vu_0)_k(a_r + c_r p^2) p \; dp$$
with $p_{\min}:= \xi_r(\lambda_{\min})$ and $p_{\max}:=
\xi_r(\lambda_{\max})$.
%
%

\bigskip

\paragraph{\it Second step: Change of the parameters $(t,x)$}
In order to get bounded parameters, we change $(t,x)$ into $(\tau,\chi)$ defined by
$$\tau=\dfrac{t}{\omega} \qquad \textrm{and} \qquad \chi=\dfrac{x}{\omega}  \qquad \textrm{with} \qquad \omega= \sqrt{t^2 + x^2},$$
following an argument from \cite{miha}. Thus the argument of the
exponential in the integral defining $u_{+}$ becomes:
$$ i \omega  (\sqrt{a_r + c_r p^2} \tau - p \chi) =: i \omega
  \varphi (p, \tau, \chi).$$
Note that $\tau, \chi \in [0,1]$ for $t,x \in [0,\infty)$.

\bigskip

%
%
\paragraph{\it Third step: Application of the stationary phase method}
Now we want to apply Theorem~\ref{horm.theo} to $u_+$ with the
amplitude $U$ and the phase $\Psi$ defined by:
\[ U(p):= q_k(a_r +
c_r p^2) (Vu_0)_k(a_r + c_r p^2) p, \
\Psi(p):= \varphi (p, \tau, \chi), \ p \in [p_{\min},p_{\max}].
\]
The functions $U$ and $\Psi$ satisfy the regularity conditions on
the compact interval $K:=[p_{\min},p_{\max}]$ and $\Psi$ is a
real-valued function.
One easily verifies, that for $\tau \neq 0$
$$\Psi'(p)= \frac{c_r p}{\sqrt{a_r+c_r p^2}} \tau - \chi =0
\  \iff \
p=p_0
:= \sqrt{\frac{a_r \chi^2}{c_r(c_r \tau^2 - \chi^2)}}
= \sqrt{\frac{a_r x^2}{c_r(c_r t^2 - x^2)}} \ ,
$$
and that this stationary point $p_0$ belongs to the interval of
integration $[p_{\min},p_{\max}]$, if and only if $(t,x)$ lies in the
cone defined by \eqref{cone-t-x}. Furthermore for $p \in \R$
$$\frac{\partial^2 \Psi}{\partial p^2}(p) = \frac{\partial^2
\varphi}{\partial p^2}(p,\tau, \chi)= \tau \frac{c_r a_r}{(a_r + c_r
p^2)^{3/2}} \not= 0 \   .$$
Thus, Theorem~\ref{horm.theo} implies that for all $(t,x)$ satisfying
\eqref{cone-t-x} there exists a constant
$C(K,\tau,\chi)>0$ such that
$$
 \Bigl|
u_+(t,x) -
e^{-i \omega  \varphi (p_0,\tau, \chi)}
\underbrace{
\left( \frac{\omega}{2 i \pi} \frac{\partial^2 \varphi }{\partial
p^2}(p_0,\tau, \chi) \right)^{-1/2}  U(p_0)
}_{(*)}
 \Bigr|
\leq C(K,\tau,\chi) \| U \|_{C^2(K)} \ \omega^{-1} \
$$
for all $\omega > 0 $.

\bigskip

%
%

\paragraph{\it Fourth step: Backsubstitution}
We must now control the dependence of $C(K,\tau,\chi)$ on the
parameters $\tau,\chi$. To this end, one has to assure that $\Psi =
\varphi(\cdot,\tau, \chi)$ stays in a bounded set in $C^4(X)$, if
$\tau$ and $\chi$ vary in $[0,1]$, where we choose $X=(p_m,p_M)$
such that $0<p_m<p_{\min}<p_{\max}<p_M<\infty$.
This follows using the above expressions for $\frac{\partial
\varphi}{\partial p}, \frac{\partial^2 \varphi}{\partial p^2}$ and
$$
\frac{\partial^3 \varphi}{\partial p^3}(p,\tau, \chi)=
-\frac{3a_r c_r^2 p}{(a_r+c_r p^2)^{5/2}} \ \tau \ ,
 \quad
\frac{\partial^4 \varphi}{\partial p^4}(p,\tau, \chi)=
-\frac{3a_r c_r^2 (a_r -4p^2 c_r)}{(a_r+c_r p^2)^{7/2}} \ \tau
$$
for $p \in X$. Thus Theorem~\ref{horm.theo} implies that there
exists a constant $C(K)>0$ such that $C(K,\tau,\chi) \leq C(K)$ for
all $\tau,\chi \in [0,1]$.
\\
To evaluate $(*)$ we observe that
$p_0 = \frac{1}{\sqrt{c_r}}\sqrt{\frac{a_r}{ c_r (t/x)^2 - 1}}$.
This implies
$$\xi_l(a_r+c_r p_0^2) = \sqrt{\frac{(a_r + c_r p_0^2) - a_l}{c_l}}
= \frac{1}{\sqrt{c_l}} \frac{\sqrt{(a_r - a_l) \left( c_r (t/x)^2 -
1 \right) + a_r}}{\sqrt{c_r (t/x)^2 - 1}}
$$
and thus
\begin{align*}
q_k(a_r+c_r p_0^2)
&= \frac{c_k \xi_k(a_r+c_r p_0^2)}{|\sum_{l=1}^n c_l \xi_l(a_r+c_r
p_0^2)|^2} \\
&= \sqrt{c_k}\sqrt{c_r (t/x)^2 - 1} \frac{\sqrt{ (a_r - a_k) \left(
c_r (t/x)^2 - 1 \right) + a_r}} {\left| \sum_{l=1}^n \sqrt{c_l}
\sqrt{ (a_r - a_l) \left( c_r (t/x)^2 - 1 \right) + a_r} \right|^2}\
.
\end{align*}
Finally,
$$
\frac{\partial^2 \varphi }{\partial p^2}(p_0,\tau, \chi)
=\tau \frac{c_r a_r}{(a_r + c_r p_0^2)^{3/2}}
=\tau \frac{\left( c_r \tau^2 - \chi^2 \right)^{3/2}}{(a_r
c_r)^{1/2} \tau^2}
= \tau (c_r a_r)^{-1/2}
\left( \frac{c_r (t/x)^2 -1 }{( t/x)^2} \right)^{3/2} .
$$

Combining these results and using $\omega \tau = t$ we find
\protect\begin{align*}
 (*) &=   \left( \frac{\omega}{2 i \pi} \ \frac{\partial^2 \varphi
}{\partial p^2}(p_0,\tau, \chi)\right)^{-1/2} \ q_k(a_r + c_r p_0^2)
\ p_0 \
(Vu_0)_k(a_r+ c_r p_0^2) \displaybreak[0]\\
 &=(2 i \pi)^{1/2} t^{-1/2} (c_r a_r)^{1/4} \left( \frac{\left( t/x
\right)^2}{c_r \left( t/x \right)^2 -1 } \right)^{3/4} \times
\\
& \quad \times \sqrt{c_k}\sqrt{c_r (t/x)^2 - 1} \frac{\sqrt{ (a_r -
a_k) \left( c_r (t/x)^2 - 1 \right) + a_r}} {\left| \sum_l
\sqrt{c_l} \sqrt{ (a_r - a_l) \left( c_r (t/x)^2 - 1 \right) + a_r}
\right|^2} \
\times
\\
& \quad  \times
\frac{1}{\sqrt{c_r}}\sqrt{\frac{a_r}{ c_r (t/x)^2 - 1}} \
(Vu_0)_k(a_r+ c_r p_0^2)
\\
  &= (2 i \pi)^{1/2}a_r^{3/4}c_r^{1/4} \  h_1(t,x) \ h_2(t,x)\
(Vu_0)_k(a_r+ c_r p_0^2) \ t^{-1/2}
\protect\end{align*}
with
\[ h_1(t,x):= \left( \frac{\left( t/x \right)^2}{c_r \left( t/x
\right)^2 -1 } \right)^{3/4}, \quad
h_2(t,x):= \frac{\sqrt{ (a_r - a_k) \left( c_r (t/x)^2 - 1 \right) +
a_r}} {\left| \sum_l \sqrt{c_l} \sqrt{ (a_r - a_l) \left( c_r
(t/x)^2 - 1 \right) + a_r} \right|^2}
\]

\bigskip

%
%
\paragraph{\it Fifth step: Uniform estimates}
It remains to estimate $h_1(t,x)  h_2(t,x) (Vu_0)_k(a_r+ c_r p_0^2)$
uniformly in $t$ and $x$, if $(t,x)$ satisfies \eqref{cone-t-x}.
To this end we note that the function $b \mapsto \dfrac{b}{c_r b -
1}$ is a decreasing function on $(1/c_r , +\infty)$. Thus the
maximum of $h_1$ for $(t,x)$ satisfying \eqref{cone-t-x}  is
attained at $\frac{t}{x} = \sqrt{\frac{\lambda_{\max}}{c_r
(\lambda_{\max}-a_r)}}$\ .  This implies
\begin{equation} \label{h1}
h_1(t,x) \leq \Bigl(\frac{1}{c_r a_r}\lambda_{\max}\Bigr)^{3/4}
\hbox{ for } (t,x) \hbox{ satisfying } \eqref{cone-t-x}.
\end{equation}
\noindent Let us now estimate $h_2$. For a fixed $v$ in
$[v_{\min},v_{\max}]$, we denote by $I_1$ the set of indices $l$
such that $(a_r - a_l)v + a_r \geq 0$ and
$I_2=\{ 1, \ldots , n \} \setminus I_1$. \\
Then, for any $v$ in $[v_{\min},v_{\max}]$ we have
$\{ 1, \ldots , r  \} \subset I_1$ (since $v_{\min}>0$) and
\begin{align*}
 \Big\vert \sum_l \sqrt{c_l}  &\sqrt{ (a_r - a_l) v + a_r} \Big\vert^2 \\
& = \Big( \sum_{l \in I_1} \sqrt{c_l} \sqrt{ (a_r - a_l) v + a_r}
\Big)^2 +\Big\vert \sum_{l \in I_2} \sqrt{c_l} \sqrt{ (a_r - a_l)
v + a_r} \Big\vert^2 \\
& \geq \Big( \sum_{l \leq r} \sqrt{c_l} \sqrt{ (a_r - a_l) v + a_r} \Big)^2 \\
& \geq \Big( \sum_{l \leq r} \sqrt{c_l} \sqrt{ (a_r - a_l) v_{\min} +
a_r} \Big)^2 \ .
\end{align*}
Thus, \eqref{cone-v} implies
\begin{equation} \label{h2}
\vert h_2(t,x)\vert \leq \frac{\max_{v \in [v_{\min},v_{\max}]}
\sqrt{|(a_r-a_k)v+a_r|}} {\left( \sum_{l\leq r} \sqrt{c_l}
\sqrt{(a_r-a_l)v_{\min} + a_r} \right)^2} \ .
\end{equation}
Putting everything together, the assertion of the theorem is valid for
\begin{equation} \label{H}
 H (t,x,u_0):=
e^{-i  \varphi (p_0,t,x)}
(2i\pi)^{1/2}a_r^{3/4}c_r^{1/4} c_k^{1/2} \  h_1(t,x) \ h_2(t,x)\
(Vu_0)_k(a_r+ c_r p_0^2) \ .
\end{equation}
Finally the right hand side of estimate \eqref{time decay} is
derived from the inequality
\begin{equation}
C(K,\tau,\chi) \| U \|_{C^2(K)} \ \omega^{-1}
\leq
C(K) \bigl\| p \mapsto
 q_k(a_r + c_r p^2)
(Vu_0)_k(a_r + c_r p^2)p
\bigr\|_{C^2([p_{\min},p_{\max}])} \ t^{-1}
\label{C}
\end{equation}
The $C^2$-norm is finite, since the involved functions are regular on
the compact set $[p_{\min},p_{\max}].$
\end{proof}

%
\section{Growing potential step}
\label{acc.decay}
\noindent For this section we specialize to the case of two branches
$N_1$ and $N_2$ and, for the sake of simplicity, we also set $c_1 =
c_2 = 1$. We show that, choosing a generic initial condition $u_0$
in a compact energy band included in $(a_2, \infty)$, the
coefficient $H(t,x,u_0)$ in the asymptotic expansion of
Theorem~\ref{sol.decay} tends to zero, if the potential step $a_2 -
a_1$ tends to infinity. Simultaneously the cone of the exact
$t^{-1/2}$-decay shrinks and inclines toward the $t$-axis.

\begin{theo} \label{sol.acc.decay}
 Let $0 < \alpha < \beta < 1$ and $\psi \in C^2_c((\alpha, \beta))$ with $\|
 \psi \|_\infty = 1$ be given. Setting $\tilde \psi (\lambda) := \psi(\lambda -
 a_2)$, we choose the initial condition $u_0 \in H$ satisfying $(V u_0)_2 \equiv
 0$ and $(V u_0)_1 = \tilde \psi$. Furthermore, let $u_+$ be defined as in
 Theorem~\ref{solform.lambda}.

 Then there is a constant $C(\psi, \alpha, \beta)$ independent of $a_1$ and
 $a_2$, such that for all $t \in \R^+$ and all $x \in N_2$ with
 \[ \sqrt{\frac{a_2 + \beta}{\beta}} \le \frac tx \le
    \sqrt{\frac{a_2 + \alpha}{\alpha}}
 \]
 the value $H(t,x,u_0)$ given in \eqref{H} satisfies
 \[ \bigl| u_+(t,x) - H(t,x,u_0) \cdot t^{-1/2} \bigr| \le C(\psi, \alpha,
    \beta) \cdot t^{-1}
 \]
 and
 \[ \bigl| H(t,x,u_0) \bigr| \le \sqrt{2\pi}
    \frac{\sqrt{\beta} (a_2 + \beta)^{3/4}}
    {\sqrt{a_2} \sqrt{a_2 - a_1 + \beta}}.
 \]
\end{theo}

\begin{proof}
 Note that it is always possible to choose the initial condition in the
 indicated way, thanks to the surjectivity of $V$,
 cf.\@ Theorem~\ref{V.iso}~ii).

 The constant $C(\psi, \alpha,
    \beta)$ has been
 already  calculated in Theorem~\ref{sol.decay}. It remains to make
 sure that it is independent of $a_1$ and $a_2$ and to prove the estimate for
 $|H(t,x,u_0)|$.

 We start with the latter and carry out a refined analysis of the proof of
 Theorem~\ref{sol.decay} for our special situation. Using the notation of
 this proof, \eqref{H} yields
 \[ \bigl| H(t,x,u_0) \bigr| = \sqrt{2\pi} a_2^{3/4} h_1(t,x) \bigl| h_2(t,x)
    \bigr| \cdot \| (Vu_0)_1 \|_\infty.
 \]
 By \eqref{h1}
 and $\lambda_{\mathrm{max}} = a_2 + \beta$ we find
 \[ h_1(t,x) \le \frac{(a_2 + \beta)^{3/4}}{a_2^{3/4}}
 \]
 and, investing the definition of $h_2$ together with \eqref{cone-v}, we have
 \begin{align*}
  \bigl| h_2(t,x) \bigr| &= \biggl|
    \frac{\sqrt{(a_2 - a_1) ((t/x)^2 - 1) + a_2}}
    {\bigl( \sqrt{(a_2 - a_1) ((t/x)^2 - 1) + a_2} + \sqrt{a_2} \bigr)^2}
    \biggr|
    \le \frac{1}{\sqrt{(a_2 - a_1) ((t/x)^2 - 1) + a_2}} \\
  &\le \frac{1}{\sqrt{(a_2 - a_1) v_{\mathrm{min}} + a_2}}.
 \end{align*}
 Putting in the definitions of $v_{\mathrm{min}}$ and afterwords
 $\lambda_{\mathrm{max}}$ and rearranging terms, this leads to
 \[ \bigl| h_2(t,x) \bigr| \le \frac{\sqrt{\beta}}
    {\sqrt{a_2} \sqrt{a_2 - a_1 + \beta}}.
 \]
 Since $\| (Vu_0)_1\|_\infty = \|\tilde \psi\|_\infty = \|\psi\|_\infty$ was set
 to $1$, we arrive at the estimate
 \[ \bigl| H(t,x,u_0) \bigr| \le \sqrt{2\pi} (a_2 + \beta)^{3/4}
    \frac{\sqrt{\beta}}{\sqrt{a_2} \sqrt{a_2 - a_1 + \beta}}.
 \]
 Going again back to Theorem~\ref{sol.decay} for the constant $C$ we have by
 \eqref{C}
 \[ C = C(K) \bigl\| U(p) \|_{C^2(K)},
 \]
 where
 \[ U(p) = p q_1(a_2 + p^2) (Vu_0)_1(a_2 + p^2), \quad p \in K,
 \]
 and
 \[ K = [p_{\mathrm{min}}, p_{\mathrm{max}}] = [\xi_2(a_2 + \alpha),
    \xi_2(a_2 + \beta)] = [\sqrt{\alpha}, \sqrt{\beta}].
 \]
 Thus, the constant $C(K)$ is independent of $a_1$ and $a_2$ and we can
 start to estimate the $C^2$-norm of $U$:
 \begin{align*}
  U(p) &= p \tilde \psi(a_2 + p^2)
    \frac{\xi_1(a_2 + p^2)}{|\xi_1(a_2 + p^2) + \xi_2(a_2 + p^2)|^2} =
    p \psi(p^2) \frac{\sqrt{a_2 - a_1 + p^2}}
    {\bigl( \sqrt{a_2 - a_1 + p^2} + p \bigr)^2} \\
  &= p \psi(p^2)
    \frac{f(p)}{(f(p) + p)^2}
 \end{align*}
 where $f(p) := \sqrt{a_2 - a_1 + p^2}$. For the function $U$ itself
 we find
 \[ | U(p) | \le \sqrt{\beta} \| \psi \|_\infty \frac{1}{f(p)} =
    \frac{\sqrt{\beta}}{\sqrt{a_2 - a_1 + p^2}} \le
    \frac{\sqrt{\beta}}{\sqrt{a_2 - a_1 + \alpha}} \le
    \frac{\sqrt{\beta}}{\sqrt{\alpha}}.
 \]
 Calculating the derivatives is lengthy, but using $f'(p) f(p) =
 p$, one finds constants $C_1$ and $C_2$ depending only on $\psi$, $\alpha$ and
 $\beta$ with
 \begin{align*}
  |U'(p)| &= \Bigl| \bigl( p \psi(p^2) \bigr)' \frac{ f(p)}{(f(p) + p)^2} +
    p \psi(p^2) \frac{ p^2 - pf(p) -2 f(p)^2}{f(p)(f(p) + p)^3} \Bigr| \\
  &\le C_1 \Bigl( \frac{1}{f(p)} +
    \frac{2p^2 + 4 p f(p) + 2 f(p)^2}{f(p) (f(p) + p)^3} \Bigr) \le
    C_1 \Bigl( \frac{1}{f(p)} + \frac{2}{f(p)^2} \Bigr)^2 \le C_1
    \Bigl( \frac{1}{\sqrt{\alpha}} + \frac{2}{\alpha} \Bigr)
 \end{align*}
 and in a similar manner
 \begin{align*}
  |U''(p)| &= \Bigl| \bigl( p \psi(p^2) \bigr)'' \frac{f(p)}{(f(p) + p)^2} +
    2 \bigl( p \psi(p^2) \bigr)'
    \frac{p^2 - pf(p) - 2 f(p)^2}{f(p) (f(p) + p)^3} \\
  & \hspace{3cm} \strut + p \psi(p^2)
    \frac{5 f(p)^4 + 8 p f(p)^3 - 4 p^3 f(p) - p^4}{f(p)^3 (f(p) + p)^4}
    \Bigr| \\
  &\le C_2 \Bigl( \frac{1}{\sqrt{\alpha}} + \frac{4}{\alpha} +
    \frac{5}{\alpha^{3/2}} \Bigr).
 \qedhere
 \end{align*}
\end{proof}
%
%
\begin{rem}
\begin{enumerate}
\item
 In the situation of Theorem~\ref{sol.acc.decay}, we have
 \[ \bigl| H(t,x,u_0) \bigr| \le \sqrt{2\pi} (a_2 + \beta)^{3/4}
    \frac{\sqrt{\beta}}{\sqrt{a_2} \sqrt{a_2 - a_1 + \beta}} \sim
    \sqrt{2 \pi \beta} \ a_2^{-1/4} \quad \text{as } a_2 \to +\infty.
 \]
\item
Suppose that $\psi(\mu) \ge m <0$ for $\mu \in [\alpha',\beta']$
with $\alpha < \alpha' < \beta' < \beta$. Then one can show that
\[
\bigl| H(t,x,u_0) \bigr| \ge  \sqrt{2 \pi \alpha} \ a_2^{-1/4} m.
 \]
for $(t,x)$ satisfying
 \[ \sqrt{\frac{a_2 + \beta'}{\beta'}} \le \frac tx \le
    \sqrt{\frac{a_2 + \alpha'}{\alpha'}}
 \]
if $a_2$ is sufficiently large. Thus the coefficient of $t^{-1/2}$
behaves exactly as const$\, \cdot \, a_2^{-1/4}$ (in particular it
tends to zero) as $a_2 \to +\infty$.
\item
The cone in the $(t,x)$-plane, where $u_+$ decays as const$\, \cdot
\, t^{-1/2}$ is given by
 \[ \sqrt{\frac{\beta}{a_2 + \beta}} \le \frac xt
 \le
    \sqrt{\frac{\alpha}{a_2 + \alpha}}.
 \]
Clearly it shrinks and inclines toward the $t$-axis as $a_2 \to
+\infty$. One can prove that outside this cone, $u_+$ decays at
least as $t^{-1}$.
This exact asymptotic behavior of the $L^{\infty}$-norm might be
experimentally verified.
\item
Note that \eqref{time decay} also implies that
\begin{align*}
\bigl|u_+(t,x)\bigr|
&\le \bigl|u_+(t,x)- H(t,x,u_0)t^{-1/2}+H(t,x,u_0)t^{-1/2} \bigr| \\
&\le C(\psi,\alpha,\beta) t^{-1} + \bigl|H(t,x,u_0) \bigr|t^{-1/2} \\
& \le D(\psi,\beta,a_1,a_2) t^{-1/2}, \ x \in \R, \ t \ge 1.
\end{align*}
\end{enumerate}
\end{rem}
%

\end{document}